\documentclass[graybox]{svmult}

\usepackage{type1cm}        
\usepackage{makeidx}         
\usepackage{graphicx}        
\usepackage{multicol}        
\usepackage[bottom]{footmisc}

\usepackage{newtxtext}       %
\usepackage[varvw]{newtxmath}       

\usepackage{enumerate}
\usepackage{algorithm}
\usepackage{fullpage}
\usepackage{verbatim}
\usepackage{multicol}

\usepackage{gensymb}
\usepackage{algorithm, algorithmicx, algpseudocode}

\DeclareMathOperator*{\argmin}{arg\,min}
\newcommand{\norm}[1]{\lVert#1\rVert}

\newcommand{\iprod}[2]{\langle#1, #2\rangle}

\algnewcommand{\initialise}[1]{%
  \State \textbf{initialise}
  \Statex \hspace*{\algorithmicindent}\parbox[t]{.8\linewidth}{\raggedright #1}
}

\makeindex             

\begin{document}

\title*{A Tree Structure Approach to Reachability Analysis}
\author{Alessandro Alla and Peter M. Dower and Vincent Liu}
	
\institute{Alessandro Alla \at Department of Molecular Sciences and Nanosystems,
	Università Ca' Foscari Venezia \\ \email{alessandro.alla@unive.it}
\and Peter M. Dower \at Department of Electrical and Electronic Engineering, The University of Melbourne\\ \email{pdower@unimelb.edu.au}
\and Vincent Liu \at Department of Electrical and Electronic Engineering, The University of Melbourne\\ \email{liuv2@studentunimelb.edu.au}}
%
%
\maketitle 

\abstract{Reachability analysis is a powerful tool when it comes to capturing the behaviour,
thus verifying the safety, of autonomous systems. However, general-purpose methods, such as
Hamilton-Jacobi approaches, suffer from the curse of dimensionality. In this paper, we mitigate this problem for systems of moderate dimension and we propose a new algorithm based on a tree structure approach with geometric pruning. The numerical examples will include a comparison with a standard finite-difference method for linear and nonlinear problems.\\
{\bf Keywords:} Reachability analysis, Hamilton-Jacobi equations, optimal control, tree structure, convex geometry}

\section{Introduction}
\label{sec:1}
The development and production of self-driving cars and the deployment of drones in industrial applications are exemplars of society's ever-increasing fascination of autonomous vehicles. Commensurate to the growing integration of these vehicles in day-to-day life is the concern regarding how safe these unmanned vehicles are. These concerns are particularly prevalent in safety-critical applications such as human-robot interactions, disaster responses, and the use of high-value machinery. In such applications, being able to characterize all possible behaviours of these autonomous vehicles would be a rigorous way to verify their safety. Reachable sets lend themselves well towards this goal. When computed forwards in time, they characterise all possible states that can be reached using a constraint admissible control from some initial set of states. Similarly, when computed backwards in time, they characterise all possible states that are able to reach a terminal set of states using a constraint admissible control.

The computation of reachable sets may be done via the Hamilton-Jacobi-Bellman (HJB) equations, i.e. one of the most powerful formal verification tools for guaranteeing performance and safety properties of systems. This approach is rather general and works for controlled nonlinear systems that involve disturbances or adversarial behaviors, and despite this, characterizes the exact reachable set rather than approximations. However, this method suffers from the curse of dimensionality and it is hard to build numerical methods for high dimensional problems. In the last two decades several contributions on the mitigitation of the curse of dimensionality have been investigated mainly for optimal control problems such as, e.g.
model order reduction \cite{KVX04, AFV17}, tree structure algorithms \cite{AFS19,AS20}, spectral methods \cite{KK18}, max-plus algebra \cite{M07,M09}, Hopf-Lax approaches \cite{CDOW19, YD21}, neural networks \cite{DLM20, DM21}, tensor decomposition \cite{DKK21, OSS22} and sparse grids method \cite{BGGK13}.

For the approximation of the reachable sets in \cite{MB05}, the authors provides a formulation that requires numerically solving a Hamilton-Jacobi partial differential equation. This is a grid-based approach which is typically limited to systems of no more
than 4 states on standard computers. Therefore, the study of higher dimensional problems remains an open research area. Under certain assumption on the system, one could  decompose it appropriately, and obtain efficient algorithms for computing reachable sets, see e.g. \cite{CHVBT18}. Other approaches have been studied in \cite{AK13} where reachable sets for nonlinear systems are computed via results on reachability for uncertain linear system. Linearisation error is explicitly accounted for in \cite{AK13} using an iterative algorithm to bound this error in an over-approximative manner. Similarly, in \cite{YA21}, uncertain linear systems are considered with a focus on producing zonotopic under-approximations of reachable sets. In their work, the representational complexity of the reachable sets grows as the algorithm iterates in time, thus motivating the use of a `pruning' step, which reduces the order of the zonotopic sets. 

In this paper, we present an algorithm based on a tree structure to approximate the HJB equation for backwards reachable sets. The idea of the algorithm is based on the paper \cite{AFS19} and approximates the value function using the Dynamic Programming Principle (DPP) on an unstructured mesh for optimal control problems. The idea of our proposed algorithm is as follows. 

We start from a discretization of the terminal set and compute the value function on those points. Then, we neglect the interior points, say those nodes for which the value function is strictly negative. This is a pruning strategy that aims to mitigate the exponential increase in the cardinality of the tree. We then evolve these nodes backwards in time, making use of a result provided in Section 4, where the forward controlled dynamical system is equivalent, up to minor changes in the problem, to the backward system. This is the main difference with respect to the method in \cite{AFS19}. We are able, in this work, to compute the tree backwards in time and to prune the tree using the information from the value function. Our pruning is based on geometric considerations where the interior of the reachable set is pruned. This comes from the observation, which is demonstrated in this work, that the boundary of the reachable set at a particular time cannot evolve from the interior of the reachable set at a prior time. Thus, it is wasteful to propagate the tree structure of \cite{AFS19} for nodes that lie interior to the reachable set. When the value function is convex, interior nodes can be easily identified using off-the-shelf algorithms for computing convex hulls. 


The outline of the paper is the following. In Section 2 we present the control problem setup and in Section 3 we provide the relevant background for the characterization of the backwards reachable set. In Section 4 we show the equivalence between backwards and forwards reachable sets. Our algorithm is introduced and discussed in Section 5. Numerical examples are then shown in Section 6. Finally, conclusions and future works are discussed in Section 7.

\section*{Notation}
\begin{itemize}
    \item Let $\mathbb{I}_n$ denote an $n$-by-$n$ identity matrix.
    \item Let $\langle \cdot, \cdot\rangle$ denote the Euclidean inner product.
    \item Let $\norm{w}$ denote any norm of a vector $w$.
    \item Let $int\left(\mathcal{A}\right), cl\left(\mathcal{A}\right)$, and $\partial \mathcal{A}$ denote the interior, closure, and boundary of a set $\mathcal{A}$, respectively.
    \item Let $\mathcal{C}^k(\Omega;\,\mathbb{R})$ denote the space of $k$-times continuously differentiable functions from $\Omega$ to $\mathbb{R}$ with $\mathcal{C} \doteq \mathcal{C}^0$.
    \item Let $\mathbb{B}_R(x_0) \doteq \{x\in\mathbb{R}^n\,|\,\norm{x - x_0} \leq R\}$ denote a ball of radius $R\geq 0$ with centre $x_0\in\mathbb{R}^n$.
    \item An ellipsoidal set with centre $q$ and shape $Q$ is defined as
\begin{equation}
    \mathcal{E}(q, Q) \doteq \left\{x\in\mathbb{R}^n\, | \, (x - q)^TQ^{-1}(x-q) \leq 1\right\},
    \label{eq: notation: ellipsoidal set def}
\end{equation}
where $q\in \mathbb{R}^n$ and $Q\in\mathbb{R}^{n\times n}$ is a symmetric, positive definite matrix. The axes of $\mathcal{E}(q,Q)$ are aligned with the eigenvectors of $Q$ with lengths along these axes being equal to the square root of the corresponding eigenvalues. 
\item Let $\text{conv}\left(\mathcal{A}\right)$ denote the convex hull of a finite set of points $\mathcal{A} = \{a_i\}_{i\in\{1,\cdots,n_p\}}$, defined by
\begin{equation}
    \text{conv}\left(\mathcal{A}\right)= \left\{\sum^{n_p}_{i=1}\lambda_i a_i \,\Bigg|\, \lambda_i \geq 0 \text{ for all } i \in \{1,\cdots,n_p\} \text{ and } \sum^{n_p}_{i=1}\lambda_i = 1\right\}.
    \label{eq: notation: conv hull def}
\end{equation}
\end{itemize}
\section{Problem Setup}
\label{sec: problem setup}
In this section we begin with a system description and review relevant background relating to optimal control. Consider the continuous-time nonlinear system described by
\begin{equation}
    \dot{x}(t) = f(x(t), u(t)), \quad \forall t \in (0, T),
    \label{eq: prob setup: nonlinear system}
\end{equation}
where $T \geq 0$, $x(t)\in\mathbb{R}^n$ is the state and $u(t)\in\mathbb{U}$ is the input at time $t$, with $\mathbb{U}\subset \mathbb{R}^m$ being compact. The input is selected such that $u \in \mathcal{U}$, where
\begin{equation}
    \mathcal{U} \doteq \left\{u:[0,T]\rightarrow \mathbb{U}\;| \;u \text{ measurable} \right\}.
    \label{eq: prob setup: admissible controls}
\end{equation}
The following flow field conditions are assumed throughout.
\begin{assumption}
The function $f:\mathbb{R}^n\times\mathbb{U}\rightarrow\mathbb{R}^n$ satisfies
\begin{enumerate}[i)]
\item $f\in\mathcal{C}\left(\mathbb{R}^n\times\mathbb{U}\,; \mathbb{R}^n\right)$; and
    \item $f$ is locally Lipschitz continuous in $x$, uniformly in $u$, i.e. for any $R > 0$, there exists a Lipschitz constant $C^f_R > 0$ such that $\norm{f(x,u)- f(y,u)} \leq C^f_R\norm{x-y}, \;\; \forall (x,y,u)\in\mathbb{B}_R(0)\times\mathbb{B}_R(0)\times\mathbb{U}$.
\end{enumerate}
\label{assumption: prob setup: flow field conditions}
\end{assumption}
Under Assumption \ref{assumption: prob setup: flow field conditions}, the system described by \eqref{eq: prob setup: nonlinear system} admits a unique and continuous solution for any initial condition $x(0) = x_0$ and fixed control $u(\cdot)\in\mathcal{U}$. We denote these solutions at time $0 \leq t \leq T$ by $\varphi(t ;0, x_0, u(\cdot))$.

Attach to \eqref{eq: prob setup: nonlinear system} the value function $V: [0, T]\times \mathbb{R}^n$ corresponding to the optimal control problem
\begin{equation}
    V(t,x) \doteq \inf_{u(\cdot)\in\mathcal{U}}J(t,x,u(\cdot)) \doteq \inf_{u(\cdot)\in\mathcal{U}}\left\{\int^T_t h(\varphi(s; t, x, u(\cdot)),u(s))\, ds + g\left(\varphi(T; t, x, u(\cdot))\right)\right\},
    \label{eq: prob setup: value function}
\end{equation}
where the running cost $h:\mathbb{R}^n\times \mathbb{U}\rightarrow \mathbb{R}$ and the terminal state cost $g:\mathbb{R}^n\rightarrow\mathbb{R}$ satisfy the assumptions below.
\begin{assumption}
The functions $h:\mathbb{R}^n\times \mathbb{U}\rightarrow \mathbb{R}$ and $g:\mathbb{R}^n\rightarrow\mathbb{R}$ satisfy
\begin{enumerate}[i)]
    \item $h\in\mathcal{C}\left(\mathbb{R}^n\times\mathbb{U};\,\mathbb{R}\right)$ and $g\in\mathcal{C}\left(\mathbb{R}^n;\,\mathbb{R}\right)$; and
    \item $h$ and $g$ are locally Lipschitz continuous in $x$ (uniformly in $u$ for $h$), i.e. for any $R > 0$, there exists a Lipschitz constant $C^h_R > 0$ such that $|h(x,u)- h(y,u)| \leq C^h_R\norm{x-y}, \;\; \forall (x,y,u)\in\mathbb{B}_R(0)\times\mathbb{B}_R(0)\times\mathbb{U}$ and there exists a $C^g_R > 0$ such that $|g(x)-g(y)| \leq C^g_R\norm{x-y}, \; \forall (x,y)\in\mathbb{B}_R(0)\times\mathbb{B}_R(0)$.
\end{enumerate}
\label{assumption: prob setup: running and terminal cost}
\end{assumption}
The value function defined by \eqref{eq: prob setup: value function} satisfies the Dynamic Programming Principle (see e.g. \cite{BCD97}), which is presented below.
\begin{theorem}
Let Assumptions \ref{assumption: prob setup: flow field conditions} and \ref{assumption: prob setup: running and terminal cost} hold. Then, for any $t\in[0,T]$, $s\in[0,t]$, and $x\in\mathbb{R}^n$, the value function $V:[0, T]\rightarrow\mathbb{R}^n$ in \eqref{eq: prob setup: value function} satisfies the Dynamic Programming Principle
\begin{equation}
V(s, x) = \inf_{u(\cdot)\in\mathcal{U}}\left\{\int^t_{s} h(\varphi(\tau; s, x, u(\cdot)),u(\tau))\, d\tau + V\left(t, \varphi(t; s, x, u(\cdot))\right)\right\}.
    \label{eq: prob setup: DPP}
\end{equation}
\label{theorem: prob setup: DPP}
\end{theorem}
The value function can be characterized via the viscosity solution of a HJB equation. A detailed discussion on viscosity solutions can be found in \cite{crandall1983viscosity}.
\begin{theorem}
Let Assumptions \ref{assumption: prob setup: flow field conditions} and \ref{assumption: prob setup: running and terminal cost} hold. Then, $V = V(t,x)$ in \eqref{eq: prob setup: value function} is the unique, and locally Lipschitz continuous viscosity solution of the HJB equation given by
\begin{equation}
\begin{split}
    -V_t + H(x, \nabla V) &= 0, \quad \forall (t,x)\in (0,T)\times\mathbb{R}^n,\\
    V(T,x) &= g(x), \quad \forall x\in\mathbb{R}^n,
\end{split}
\label{eq: prob setup: HJB PDE}
\end{equation}
where the Hamiltonian $H:\mathbb{R}^n\times\mathbb{R}^n\rightarrow\mathbb{R}$ is given by
\begin{equation}
    H(x,p)\doteq \max_{u\in\mathbb{U}}\left\{-\langle p, f(x, u)\rangle - h(x, u)\right\}.
    \label{eq: prob setup: HJB Hamiltonian}
\end{equation}
\label{theorem: prob setup: HJB PDE General}
\end{theorem}
\section{Backwards Reachability}
Let us now consider a special case of the optimal control problem in \eqref{eq: prob setup: value function}, which can be used to characterize a \emph{backwards} reachable set. This set characterizes all states $x\in\mathbb{R}^n$ for which there exists an admissible control $u(\cdot)\in\mathcal{U}$ leading to some terminal set of states
\begin{equation}
    \mathcal{X}_T \doteq \left\{x\in\mathbb{R}^n \, | g(x) \leq 0 \right\},
    \label{eq: prob setup: initial and terminal set}
\end{equation}
in time $T \geq 0$, where $g:\mathbb{R}^n\rightarrow\mathbb{R}$ is a bounded and locally Lipschitz continuous function. The backwards reachable set is more precisely defined below. 
\begin{definition} The backwards reachable set at time $T \geq 0$ is defined as the set 
\begin{equation}
    \mathcal{G}(T) \doteq \left\{x\in\mathbb{R}^n \, | \, \exists\, u(\cdot)\in\mathcal{U} \text{ such that } \varphi(T; 0, x, u(\cdot)) \in \mathcal{X}_T\right\},
    \label{eq: prob setup: BRS}
\end{equation}
where $\varphi(t; 0, x_0, u(\cdot))$ denotes solutions of \eqref{eq: prob setup: nonlinear system} at time $t$ from an initial state $x(0) = x_0$ with an admissible control $u(\cdot)\in\mathcal{U}$, and $\mathcal{X}_T$ is the terminal set described by \eqref{eq: prob setup: initial and terminal set}. 
\label{def: prob setup: BRS}
\end{definition}
The backwards reachable set can be characterized via the viscosity solution of the HJB equation given in \eqref{eq: prob setup: HJB PDE}--\eqref{eq: prob setup: HJB Hamiltonian}, which is described in the following theorem (see e.g. \cite{BCD97}).
\begin{theorem}
Let Assumption \ref{assumption: prob setup: flow field conditions} hold for \eqref{eq: prob setup: nonlinear system} and let the function $g:\mathbb{R}^n\rightarrow\mathbb{R}$, which defines the terminal set $\mathcal{X}_T$ in \eqref{eq: prob setup: initial and terminal set}, satisfy Assumption \ref{assumption: prob setup: running and terminal cost}. Let $v\in\mathcal{C}([0, T]\times\mathbb{R}^n;\,\mathbb{R})$ be the unique and locally Lipschitz continuous viscosity solution of the HJB equation given by
\begin{equation}
    \begin{split}
    -v_t + H(x, \nabla v) &= 0, \quad \forall (t,x)\in (0,T)\times\mathbb{R}^n,\\
    v(T,x) &= g(x), \quad \forall x\in\mathbb{R}^n,
    \end{split}
    \label{eq: prob setup: HJB PDE BRS}
\end{equation}
where the Hamiltonian $H:\mathbb{R}^n\times\mathbb{R}^n\rightarrow\mathbb{R}$ is given by
\begin{equation}
    H(x,p)\doteq \max_{u\in\mathbb{U}}\left\{-\langle p, f(x, u)\rangle \right\}.
    \label{eq: prob setup: HJB Hamiltonian for BRS}
\end{equation}
Then, the backwards reachable set $\mathcal{G}(T)$ for \eqref{eq: prob setup: nonlinear system} is
\begin{equation}
    \mathcal{G}(T) = \left\{x\in\mathbb{R}^n \; | \; v(0, x) \leq 0\right\}.
    \label{eq: prob setup: BRS via HJB zero sublevel set}%
\end{equation} \label{theorem: prob setup: HJB equation for BRS}
\end{theorem}

Next, we will demonstrate that the boundary of the backwards reachable set $\mathcal{G}(T)$ can only reach the boundary of the terminal set $\mathcal{X}_T$. This result will motivate the use of a geometric pruning criterion that will be introduced in our proposed algorithm for computing $\mathcal{G}(T)$.
\begin{lemma} Let $\mathcal{G}(T)$ denote the backwards reachable set of \eqref{eq: prob setup: nonlinear system} as defined in \eqref{eq: prob setup: BRS}. Then,
\begin{equation}
    \partial \mathcal{G}(T) \subseteq \left\{x\in\mathbb{R}^n \, | \, \exists\, u(\cdot)\in\mathcal{U} \text{ such that } \varphi(T; 0, x, u(\cdot)) \in \partial\mathcal{X}_T\right\} \label{eq: prob setup: boundary mapping}.
\end{equation}
Moreover, there does not exist a control $u(\cdot)\in\mathcal{U}$ such that $\varphi(T; 0, x, u(\cdot))\in int\left(\mathcal{X}_T\right)$ for any $x \in \partial \mathcal{G}(T)$.
\label{lemma: prob setup: boundary mapping}
\end{lemma}
\begin{proof} Let us denote the right hand side of \eqref{eq: prob setup: boundary mapping} by $\Bar{\mathcal{B}}$. We will first show that the boundary of $\mathcal{G}(T)$ can not reach the interior of $\mathcal{X}_T$. Suppose this is not true. That is, there exists $\hat{x}\in\partial\mathcal{G}(T)$ and $\hat{u}(\cdot)\in\mathcal{U}$ such that $\varphi(T;0,\hat{x},\hat{u}(\cdot)) \in int\left(\mathcal{X}_T\right)$. Under Assumption \ref{assumption: prob setup: flow field conditions}, $\varphi$ is continuous in $x$, thus there exists a sufficiently small neighbourhood $\mathcal{N}$ of $\hat{x}$ such that $\varphi(T;0, \Bar{x}, \hat{u}(\cdot)) \in int\left(\mathcal{X}_T\right)$ for all $\Bar{x}\in\mathcal{N}$. This would then imply $\mathcal{N}\subset \mathcal{G}(T)$, but this leads to a contradiction as $\hat{x}$ lies on the boundary of $\mathcal{G}(T)$, i.e. there exists a $\Bar{x}\in\mathcal{N}\setminus\mathcal{G}(T)$ such that $\varphi(T;0,\Bar{x}, \hat{u}(\cdot)) \in int\left(\mathcal{X}_T\right)$. Furthermore, since $\hat{x}\in\partial \mathcal{G}(T)\subset \mathcal{G}(T)$, there exists a $u(\cdot)\in\mathcal{U}$ such that $\varphi(T;0,\hat{x},u(\cdot))\in\mathcal{X}_T$. Thus, for all $\hat{x}\in\partial \mathcal{G}(T)$, there exists a $u(\cdot)\in\mathcal{U}$ such that $\varphi(T;0,\hat{x},u(\cdot))\in\mathcal{X}_T\setminus int\left(\mathcal{X}_T\right)=\partial\mathcal{X}_T$, and hence $\partial \mathcal{G}(T) \subseteq \Bar{\mathcal{B}}$. 
\end{proof}

\section{Forwards Reachability}
An analogous set to the backwards reachable set is the forwards reachable set, which characterizes all states that can be reached from some initial set of states under the influence of an admissible control. We will verify that the forwards reachable set of a time-reversed system is exactly the backwards reachable set of \eqref{eq: prob setup: nonlinear system}. This will prove to be convenient for us as the algorithm in Section \ref{sec: tree-based algorithm} is more intuitively described when considering a forwards reachability problem. The aforementioned time-reversed system is described by
\begin{equation}
    \dot{x}(t) = -f(x(t), u(t)), \quad \forall t\in(0,T), 
    \label{eq: prob setup: reversed nonlinear system}
\end{equation}
where $T \geq 0$, $f:\mathbb{R}^n\times\mathbb{U}\rightarrow\mathbb{R}^n$ satisfies Assumption \ref{assumption: prob setup: flow field conditions}, and $u(\cdot)\in\mathcal{U}$. When referring to the forwards reachable set, we will use $\mathcal{X}_0$ defined identically to $\mathcal{X}_T$ in \eqref{eq: prob setup: initial and terminal set} to describe the initial set of states. The forwards reachable set for \eqref{eq: prob setup: reversed nonlinear system} is now precisely defined below.
\begin{definition}The forwards reachable set of the time-reversed system in \eqref{eq: prob setup: reversed nonlinear system} at time $T \geq 0$ is defined as
\begin{equation}
    \mathcal{F}_-(T) \doteq \left\{x\in\mathbb{R}^n \, | \, \exists\, u(\cdot)\in\mathcal{U} \text{ and } \exists\, x_0 \in \mathcal{X}_0 \text{ such that } \varphi_{-}(T; 0, x_0, u(\cdot)) = x\right\},
    \label{eq: prob setup: FRS}
\end{equation}
where $\varphi_{-}(t; 0, x_0, u(\cdot))$ denotes the unique solutions of \eqref{eq: prob setup: reversed nonlinear system} at time $t$ from an initial state $x(0) = x_0$ with an admissible control $u(\cdot)\in\mathcal{U}$, and $\mathcal{X}_0$ is the initial set. 
\label{def: prob setup: FRS}
\end{definition}

To demonstrate that $\mathcal{F}_{-}(T)$ is exactly the backwards reachable set $\mathcal{G}(T)$, we formally verify a standard result relating solutions of \eqref{eq: prob setup: nonlinear system} to \eqref{eq: prob setup: reversed nonlinear system}. In particular, we wish to show that an initial state $x$ that evolves under a control $u(\cdot)$ for \eqref{eq: prob setup: nonlinear system} can be recovered by setting the terminal state of \eqref{eq: prob setup: nonlinear system} as the initial state for \eqref{eq: prob setup: reversed nonlinear system} and applying a time-reversed control $u_-(\cdot)$. It follows naturally that the other direction would hold as well. This is more precisely described below.
\begin{lemma} 
Let $x\in\mathbb{R}^n$ and $u(\cdot)\in\mathcal{U}$. Select $u_-(\cdot)\in\mathcal{U}$ such that $u_-(t) = u(T-t)$ for all $t\in[0, T]$, then, for any $T \geq 0$ we have
\begin{align}
    &\varphi_{-}(T; 0, \varphi(T; 0, x, u(\cdot)), u_-(\cdot)) = x, \label{eq: prob setup: lemma reversing forward dynamics}\\
    \text{and}\quad & \varphi(T; 0, \varphi_{-}(T; 0, x, u_-(\cdot)), u(\cdot)) = x\label{eq: prob setup: lemma forwarding reverse dynamics},
\end{align}
where $\varphi(t_2; t_1, x_1, u(\cdot))$ and $\varphi_-(t_2; t_1, x_1, u(\cdot))$ denote solutions of \eqref{eq: prob setup: nonlinear system} and \eqref{eq: prob setup: reversed nonlinear system}, respectively, at time $t_2$ with initial state $x(t_1) = x_1$ and control input $u(\cdot)$.
\label{lemma: prob setup: forward and backward solutions}
\end{lemma}
\begin{proof}
Since $\varphi(\cdot; 0, x, u(\cdot))$ is a solution to \eqref{eq: prob setup: nonlinear system} with an initial condition $x(0) = x$ under the control $u(\cdot)\in\mathcal{U}$, we have that $\frac{\partial}{\partial s}\varphi(s; 0, x, u(\cdot)) = f(\varphi(s; 0,x, u(\cdot)), u(s))$ for all $s\in(0,T)$.
Substituting $s = T-t$, we obtain $\frac{\partial}{\partial t}\varphi(T-t; 0, x, u(\cdot)) = - f(\varphi(T-t; 0, x, u(\cdot)), u_-(t))$ for all $t\in(0,T)$. Then, taking the initial condition of the reversed system \eqref{eq: prob setup: reversed nonlinear system} to be $x(0) = \varphi(T; 0, x, u(\cdot))$ and the control input to be $u_-(\cdot)\in\mathcal{U}$ where $u_-(t) = u(T-t), \; \forall t\in[0,T]$, we must have that
\begin{equation}
    \varphi_{-}(t; 0, \varphi(T; 0, x, u(\cdot)), u_-(\cdot)) = \varphi(T-t; 0, x, u(\cdot)),
    \label{eq: prob setup: lemma proof 1}
\end{equation}
for all $t\in[0, T]$, since \eqref{eq: prob setup: reversed nonlinear system} admits \emph{unique} solutions only. Letting $t = T$ in \eqref{eq: prob setup: lemma proof 1} makes the right-hand side become $\varphi(0;0,x,u(\cdot)) = x$, which yields \eqref{eq: prob setup: lemma reversing forward dynamics}. Finally, by alternating the signs between \eqref{eq: prob setup: nonlinear system} and the time-reversed system \eqref{eq: prob setup: reversed nonlinear system}, we obtain \eqref{eq: prob setup: lemma forwarding reverse dynamics} from \eqref{eq: prob setup: lemma reversing forward dynamics}. 
\end{proof}

\begin{theorem}
Let $\mathcal{F}_{-}(T)$ denote the forwards reachable set defined in \eqref{eq: prob setup: FRS} for the time-reversed system in \eqref{eq: prob setup: reversed nonlinear system}. Let $\mathcal{G}(T)$ be the backwards reachable set as defined in \eqref{eq: prob setup: BRS}. Then,
\begin{equation}
    \mathcal{G}(T) = \mathcal{F}_{-}(T).
\end{equation}\label{theorem: prob setup: BRS from FRS}
\end{theorem}
\vspace*{-0.6cm}
\begin{proof}
Let $x\in \mathcal{F}_{-}(T)$, which means for some $u_-(\cdot)\in\mathcal{U}$ and $x_0 \in \mathcal{X}_0$, we have $\varphi_{-}(T; 0, x_0, u_-(\cdot)) = x$. From Lemma \ref{lemma: prob setup: forward and backward solutions}, we have that $\varphi(T; 0, \varphi_{-}(T; 0, x_0, u_-(\cdot)), u(\cdot)) = x_0 \in \mathcal{X}_0 = \mathcal{X}_T$, where $u(t) \doteq u_-(T-t)$ for all $t\in[0, T]$. Since  $u(\cdot)\in\mathcal{U}$ takes the state $\varphi_{-}(T; 0, x_0, u_-(\cdot))=x$ to $x_0\in\mathcal{X}_T$ for the system in \eqref{eq: prob setup: nonlinear system}, we have $x \in \mathcal{G}(T)\implies \mathcal{F}_{-}(T)\subseteq\mathcal{G}(T)$.

To demonstrate that the other direction holds, let $x\in \mathcal{G}(T)$, which means for some $u(\cdot)\in\mathcal{U}$ we have $x_T \doteq \varphi(T; 0, x, u(\cdot))\in\mathcal{X}_T = \mathcal{X}_0$. From Lemma \ref{lemma: prob setup: forward and backward solutions}, we have $\varphi_{-}(T; 0, \varphi(T; 0, x, u(\cdot)), u_-(\cdot)) = x$ where $u_-(t) = u(T-t)$ for all $t\in[0, T]$. Since $u_-(\cdot)\in\mathcal{U}$, this means $x$ can be reached under the influence of an admissible control from $\varphi(T; 0, x, u(\cdot))=x_T \in \mathcal{X}_0$, thus $x\in\mathcal{F}_{-}(T)\implies \mathcal{G}(T) \subseteq \mathcal{F}_-(T)$.  
\end{proof}

Likewise with the characterization of the backwards reachable set in Theorem \ref{theorem: prob setup: HJB equation for BRS}, the forwards reachable set $\mathcal{F}_-(T)$ for \eqref{eq: prob setup: reversed nonlinear system} can be also be characterized via the solution of a corresponding HJB PDE.
\begin{corollary} 
Let Assumption \ref{assumption: prob setup: flow field conditions} hold for \eqref{eq: prob setup: reversed nonlinear system} and let the function $g:\mathbb{R}^n\rightarrow\mathbb{R}$, which defines the initial set $\mathcal{X}_0 = \mathcal{X}_T$ in \eqref{eq: prob setup: initial and terminal set}, satisfy Assumption \ref{assumption: prob setup: running and terminal cost}. Let $w\in\mathcal{C}([0, T]\times\mathbb{R}^n;\,\mathbb{R})$ be the unique and locally Lipschitz continuous viscosity solution of the HJB equation given by
\begin{equation}
\begin{split}
    w_t - H(x, \nabla w) &= 0, \quad \forall (t,x)\in (0,T)\times\mathbb{R}^n,\\
    w(0,x) &= g(x), \quad \forall x\in\mathbb{R}^n,
\end{split}
\label{eq: prob setup: HJB PDE FRS}
\end{equation}
where the Hamiltonian $H:\mathbb{R}^n\times \mathbb{R}$ is given by
\begin{equation}
    H(x,p)\doteq -\max_{u\in\mathbb{U}}\langle p, -f(x, u)\rangle = \min_{u\in\mathbb{U}}\langle p, f(x, u)\rangle.
    \label{eq: prob setup: HJB FRS Hamiltonian}
\end{equation}
Then, the forwards reachable set for \eqref{eq: prob setup: reversed nonlinear system} is
\begin{equation}
    \mathcal{F}_-(T) = \left\{x\in\mathbb{R}^n \; | \; w(T, x) \leq 0\right\}.
    \label{eq: prob setup: FRS via HJB zero sublevel set}%
\end{equation} \label{corollary: prob setup: HJB equation for FRS}
\end{corollary}
\begin{proof} For the proof, we make use of the test function form of the definition for viscosity solutions in (5.17) of \cite{CS04}. Let $w\in\mathcal{C}([0, T]\times\mathbb{R}^n;\, \mathbb{R})$ be the viscosity solution of \eqref{eq: prob setup: HJB PDE FRS} and let $v\in\mathcal{C}([0, T]\times\mathbb{R}^n;\,\mathbb{R})$ be defined such that $v(t, x) \doteq w(T-t, x), \, \forall (t,x) \in [0, T]\times\mathbb{R}^n$. Suppose for some $\xi \in \mathcal{C}^1((0,T)\times\mathbb{R}^n; \mathbb{R})$ we have that $v - \xi$ attains a local maximum at $(t_0, x_0) \in (0,T)\times\mathbb{R}^n$. Define now $\Bar{\xi} \in \mathcal{C}^1((0,T)\times\mathbb{R}^n; \mathbb{R})$ such that $\Bar{\xi}(t,x)\doteq \xi(T-t,x) ,\forall (t,x)\in[0,T]\times\mathbb{R}^n$. Then, at $(T-t_0, x_0)$ we must also have that $w - \Bar{\xi}$ attains a local maximum. Since $w$ being a viscosity solution of \eqref{eq: prob setup: HJB PDE FRS} implies it is also a viscosity subsolution of \eqref{eq: prob setup: HJB PDE FRS}, we have that
\begin{equation}
    \Bar{\xi}_t(T-t_0, x_0) + \max_{u\in\mathbb{U}}\langle \nabla\Bar{\xi}(T-t_0,x_0), -f(x_0, u)\rangle \leq 0.
    \label{eq: prob setup: proof of forward HJB 1}
\end{equation}
Substituting $\xi$ into \eqref{eq: prob setup: proof of forward HJB 1} yields
\begin{equation}
    -\xi_t(t_0, x_0) + \max_{u\in\mathbb{U}}\langle \nabla\xi(t_0,x_0), -f(x_0, u)\rangle \leq 0,
    \label{eq: prob setup: proof of forward HJB 2}
\end{equation}
which implies $v$ is a viscosity subsolution of \eqref{eq: prob setup: HJB PDE BRS}. The same procedure can be used to show that $v$ is also a viscosity supersolution of \eqref{eq: prob setup: HJB PDE BRS} by considering a candidate function $\xi\in\mathcal{C}^1((0,T)\times\mathbb{R}^n; \mathbb{R})$ such that $v - \xi$ attains a local \emph{minimum} at $(t_0, x_0)\in(0,T)\times\mathbb{R}^n$. Thus, $v$ must be a viscosity solution of \eqref{eq: prob setup: HJB PDE BRS} if $w$ is a viscosity solution of \eqref{eq: prob setup: HJB PDE FRS}. From Theorem 3, this implies that 
\begin{equation}
    \left\{x\in\mathbb{R}^n \; | \; w(T, x) \leq 0\right\} = \left\{x\in\mathbb{R}^n \; | \; v(0, x) \leq 0\right\} = \mathcal{G}(T).
\end{equation}
Theorem \ref{theorem: prob setup: BRS from FRS} can then be used to conclude \eqref{eq: prob setup: FRS via HJB zero sublevel set}. The Lipschitz and uniqueness properties follow from the properties of $v$ in Theorem \ref{theorem: prob setup: HJB equation for BRS}.
\end{proof}


\section{A Tree-based Algorithm for Computing Reachable Sets}
\label{sec: tree-based algorithm}
In this section we provide the description of our new algorithm to approximate solutions of \eqref{eq: prob setup: HJB PDE} with zero running cost, i.e. $h=0$. Our algorithm is based on a tree structure as proposed in \cite{AFS19} that is adapted to our problem for computing the backwards reachable set as defined in \eqref{eq: prob setup: BRS}. This tree structure is depicted in Figure \ref{fig: algorithm: original DP tree structure}. In \cite{AFS19}, the value function is computed by first constructing a tree, which represents a discretization in time and space of the forwards reachable set from a \emph{known} initial state denoted by $x^0_1$ in the diagram. Using the Dynamic Programming Principle (DPP) stated in Theorem \ref{theorem: prob setup: DPP}, the value function is then computed backwards in time starting from the nodes in the tree corresponding to some terminal time $t= T$. However, when it comes to computing the backwards reachable set, we want to compute the set of all initial states such that a \emph{known} terminal set $\mathcal{X}_T$ can be reached. Thus, a key difference between our approach and the algorithm proposed in \cite{AFS19} is that we build the tree backwards in time starting from time $T$ at the terminal set. We now move to describing how the algorithm of \cite{AFS19} could be adjusted to compute the backwards reachable set and state its connection to the value function in \eqref{eq: prob setup: value function}. A geometric condition for pruning is then introduced to reduce the computational expense. 

\begin{figure}
    \centering
    \begin{minipage}[r]{0.75\textwidth}
        \centering
        \includegraphics[width = 1\textwidth]{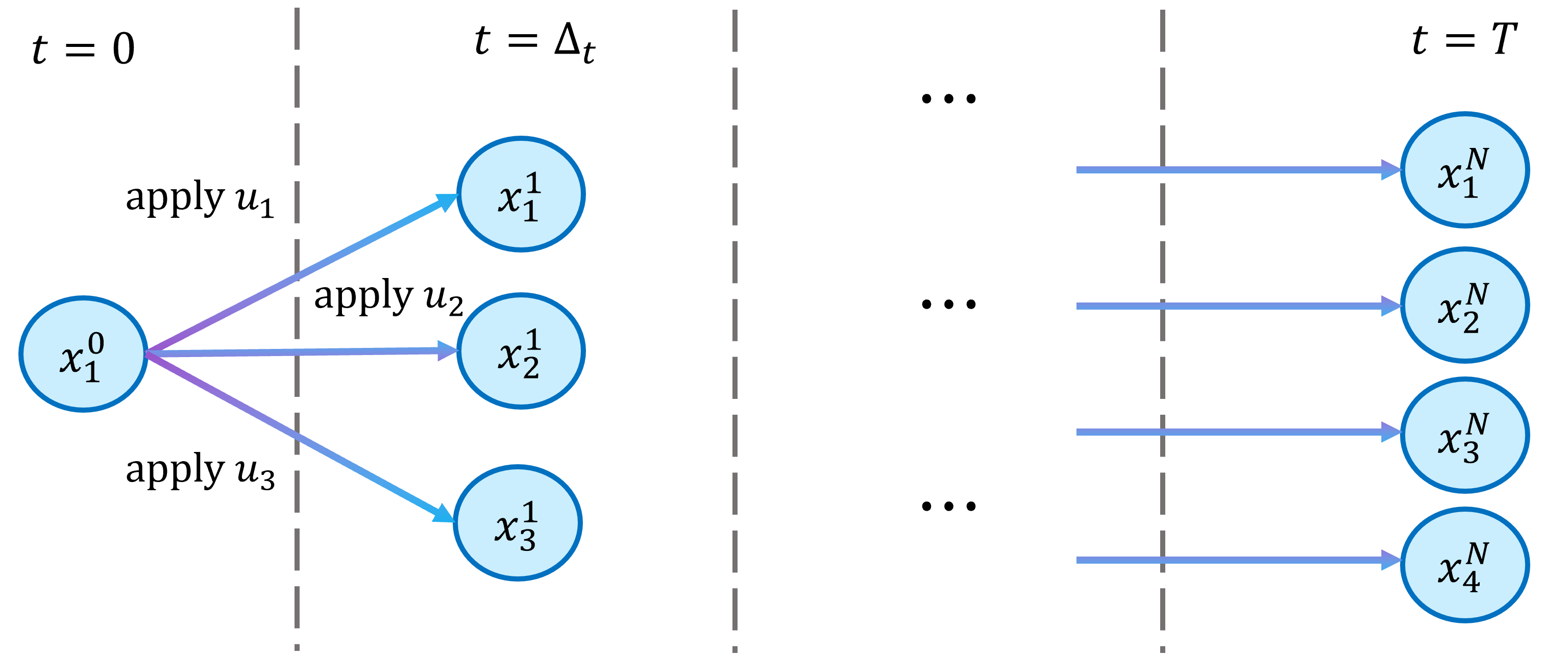}
    \end{minipage}
        \begin{minipage}[r]{0.9\textwidth}
        \centering
        $\quad\,$\includegraphics[width = 0.925\textwidth]{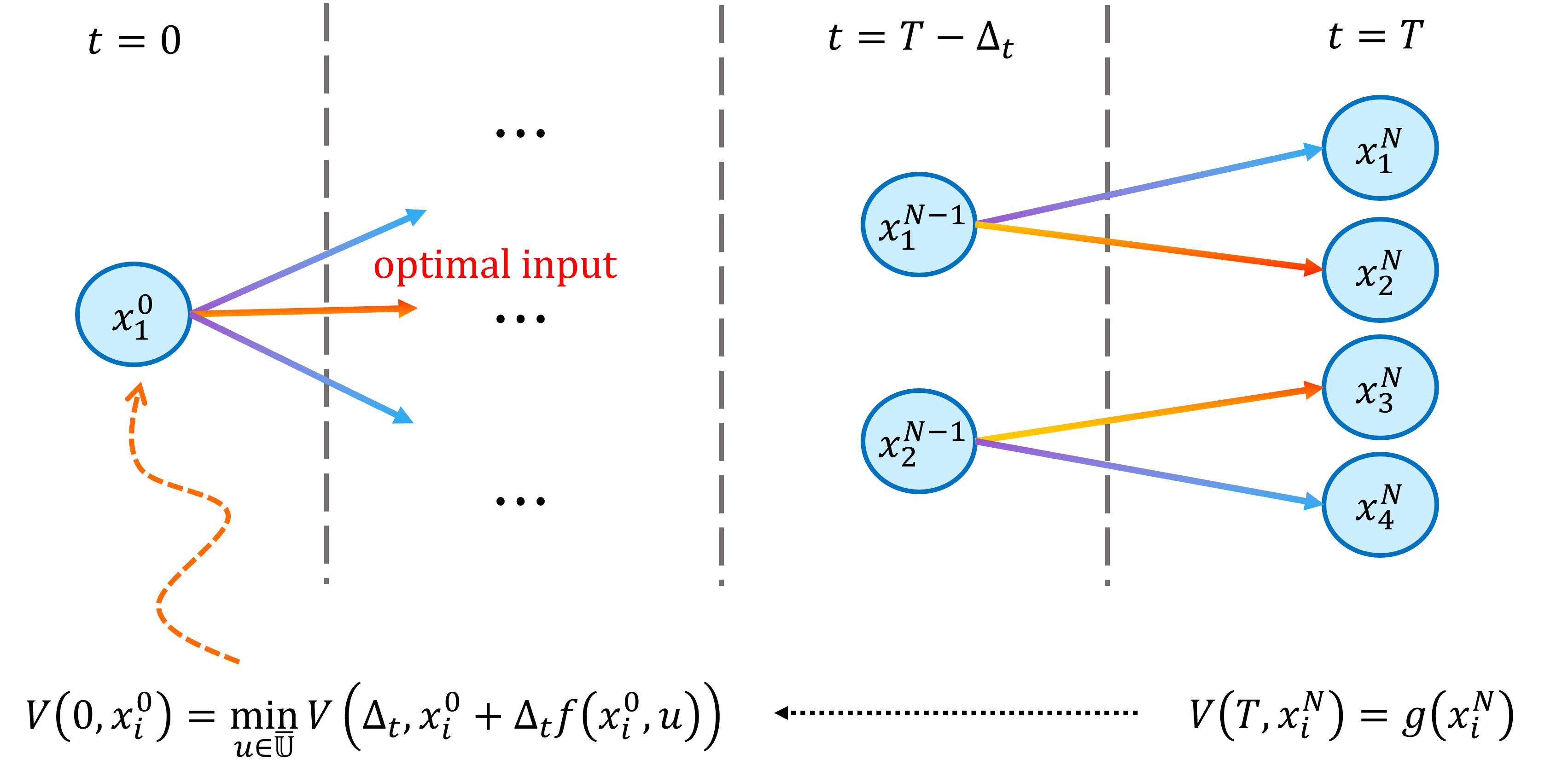}
    \end{minipage}
    \centering
    \caption{Diagram depicting tree structure and value function computation as described in \cite{AFS19}. Upper figure depicts generation of nodes in each level of the tree from a discretized input set $\mathbb{U}$. Lower figure depicts computation of the value function starting from nodes at the terminal time $T$. }
    \label{fig: algorithm: original DP tree structure}
\end{figure}

To begin, the control set $\mathbb{U}$ is discretized into a finite set $\Bar{\mathbb{U}} \doteq\{u_1, \ldots, u_{n_u}\}\subset \mathbb{U}$. Let $\Delta_t$ denote the time discretization interval with $t_k \doteq T - k\Delta_t$ being the $k$-th time point for a total of $N\doteq \lceil \textstyle{\frac{T}{\Delta_t}}\rceil$ time points. We define the $k$-th tree level $\mathcal{T}^k$ as the set of nodes generated at time $t_k$, which has a cardinality denoted by $n_k$. Additionally, we use $x^k_{i}$ to denote the $i$-th node of the $k$-th tree level. 

Then, we start to compute the nodes of our tree. The initial (zeroth) level contains nodes given by a discretization of the terminal set $\mathcal{X}_T$. Let $\{x^0_i\}_{i\in\{1,\ldots,n_0\}}$ be a finite set of points in $\mathcal{X}_T$. For these points, we can directly compute the value function as follows:
\begin{equation}
    V(T, x^0_i) = g(x^0_i), \qquad \forall i\in\{1,\ldots n_0\},
    \label{eq: algorithn: value function at terminal time}
\end{equation}
which is given by the terminal condition of \eqref{eq: prob setup: HJB PDE BRS} at time $T$. Equivalently, we can consider this step as initialising the value function $w$ in \eqref{eq: prob setup: HJB PDE FRS} for the forwards reachable set of the time-reversed system in \eqref{eq: prob setup: reversed nonlinear system}.

We continue generating the nodes of the tree for level $1$ corresponding to the time $t_1 = T-\Delta_t$.  The set of nodes $\mathcal{T}^1 \doteq \{x^1_i\}_{i\in\{1,\ldots,n_1\}}$ will be computed from the time-reversed ODE \eqref{eq: prob setup: reversed nonlinear system}. Although other discretizations are possible, an Euler discretization has been chosen here leading to:
    \begin{equation}
        x^1_{i_j} = x^0_i - \Delta_tf(x^0_i, u_j), \qquad \forall i\in\{1,\ldots,n_0\}, \;\forall j\in\{1,\ldots,n_u\}.
    \end{equation}
    Once the nodes in $\mathcal{T}^1$ are generated we approximate the value function at time $T-\Delta_t$ using a one-step discretisation of the DPP (rewritten from \eqref{eq: prob setup: DPP}) using the discretized input set $\Bar{U}$: 
    \begin{equation}
    \label{eq:DPP}
        V(T-\Delta_t, x^1_i) = \min_{u\in\{u_1,\ldots,u_{n_u}\}}V(T,x^1_i + \Delta_tf(x^1_i, u)),\qquad \forall i\in\{1,\ldots,n_1\}.
    \end{equation}
    The DPP in \eqref{eq:DPP} approximates the value function for a point $x^1_i$ at time $T - \Delta_t$ via a minimization of the value function at time $T$ across all points that can be reached by $x^1_i$ \emph{forwards} in time. However, the value function at time $T$ is only computed in \eqref{eq: algorithn: value function at terminal time} on the finite set of points $\mathcal{T}^0 = \left\{x^0_i\right\}_{i\in\{1,\ldots,n_0\}}$, which may not contain the point $x^1_i + \Delta_tf(x^1_i, u_j)$. In general, \eqref{eq:DPP} may require an interpolation of the value function.
    
    In order to compute an approximation of the value function using pre-computed values of nodes on previous tree levels, we can instead perform the minimization in \eqref{eq:DPP} over all nodes in $\mathcal{T}^0$ that can be reached by $x^1_i$. To be more precise, we consider an Euler discretisation of the forward system \eqref{eq: prob setup: nonlinear system}:
\begin{equation}
    x[n+1] = x[n] + \Delta_t f(x[n], u[n]), \qquad \forall n \in \{0, 1, \ldots\},
    \label{eq: algorithm: discretised forward system}
\end{equation}
where $x[n]\approx x(T-t_n),$ $u[n]\approx u(T-t_n) \in \mathbb{U},\; \forall n \in \{0, 1, \ldots\}$. The set of points that can be reached in a single time-step $\Delta_t$ is defined precisely below. 
\begin{definition} The one-step reachable set for the discrete-time system \eqref{eq: algorithm: discretised forward system} from a point $\Bar{x}$ is defined as
\begin{equation}
    \mathcal{R}_1(\Bar{x}) \doteq \left\{x\in\mathbb{R}^n\,|\,\exists u \in \mathbb{U} \text{ such that } \Bar{x} + \Delta_t f(\Bar{x}, u) = x\right\}.
\end{equation}
\end{definition}
From this, we construct a general iteration step for generating nodes on all tree levels $k\in\{1,\ldots,n_{k}\}$, as well as approximating the value function at these nodes:
\begin{align}
    x^k_{i_j} &= x^{k-1}_i - \Delta_tf(x^{k-1}_i, u_j),\qquad \forall i\in\{1,\ldots,n_{k-1}\}, \; \forall j\in\{1,\ldots,n_u\},\label{eq: algorithm: node generation}\\
     V(t_k, x^k_{i_j}) &= \min_{x\in \mathcal{T}^{k-1}\cap \mathcal{R}_1(x^k_{i_j})}V(t_{k-1},x),\qquad \forall i_j\in\{1,\ldots,n_k\}. \label{eq: algorithm: value function propagation}
\end{align}
 In \eqref{eq: algorithm: node generation}, we generate nodes contained in the forwards reachable set of the time-reversed system in \eqref{eq: prob setup: reversed nonlinear system} starting from an initial set $\mathcal{X}_0 = \mathcal{X}_T$. From Theorem \ref{theorem: prob setup: BRS from FRS}, this gives us the backwards reachable set of \eqref{eq: prob setup: nonlinear system}. The value function is then approximated via a minimization over a subset of nodes in the previous tree level. Unlike in Figure \ref{fig: algorithm: original DP tree structure}, the entire tree does not need to be generated a priori. Algorithm \ref{Alg:1} summarises the procedure outlined above. 
 
\begin{remark}
To check whether a point $x\in\mathcal{T}^{k-1}$ sits within the one-step reachable set $\mathcal{R}_1(x^k_{i_j})$ as required by \eqref{eq: algorithm: value function propagation}, we can consider a root finding problem of the form 
\begin{equation}
    x^k_{i_j} - x + \Delta_tf(x^k_{i_j}, u) = 0.
    \label{eq: algorithm: root finding problem}
\end{equation}
If \eqref{eq: algorithm: root finding problem} can be solved for some $u\in\mathbb{U}$, then $x$ must be contained in $\mathcal{R}_1(x^k_{i_j})$. Alternatively, theory on computing $N$-step reachable sets for linear systems is well-established (see e.g. Chapter 10 of \cite{BB17}) and can be used here to compute $\mathcal{R}_1(x^k_{i_j})$ if the system of interest is linear. Note that there must always exist an $x\in\mathcal{T}^{k-1}$ that sits within $\mathcal{R}_1(x^k_{i_j})$ since the input $u_j$, which was used to generate the node $x^k_{i_j}$ can be used to reach the node $x^{k-1}_i \in \mathcal{R}_1(x^k_{i_j})$. In particular, $\mathcal{T}^{k-1}\cap\mathcal{R}_1(x^k_{i_j})$ is non-empty.
\end{remark}

\begin{algorithm}
\caption{Tree Structure Algorithm for Computing Backwards Reachable Sets}
\label{Alg:1}
\begin{algorithmic}[1]
\Require{$\Delta_t, \Bar{\mathbb{U}} \doteq\{u_1, \ldots, u_{n_u}\}\subset \mathbb{U}$.} 
\initialise{ Discretize $\mathcal{X}_T$ to $\{x^0_i\}_{i\in\{1,\ldots,n_0\}}$ \\
$ V(T, x^0_i) = g(x^0_i).$}
\For{$k=1, \ldots, N$}
\State $x^k_{i_j} = x^{k-1}_i- \Delta_t f(x^{k-1}_i, u_j),\qquad \forall i\in\{1,\ldots,n_{k-1}\},\; \forall j\in\{1,\ldots,n_u\}$
\State $V(t_k, x^k_{i_j}) = \min_{x\in \mathcal{T}^{k-1}\cap \mathcal{R}_1(x^k_{i_j})}V(t_{k-1},x),\qquad \forall i_j\in\{1,\ldots,n_k\}$
\EndFor
\end{algorithmic}
\end{algorithm}

As the number of nodes grows by a factor of $n_u$ in each tree level $k$ of Algorithm \ref{Alg:1}, a pruning strategy becomes necessary to mitigate the exponential increase in the cardinality of the tree. In \cite{AFS19}, a pruning strategy based on the distance of the nodes on the same level of the tree was implemented. This turned out to be efficient due to the fact that the value function is a Lipschitz continuous function. Here, in our problem, we offer an alternative pruning criterion based on storing only the nodes that lie along the boundary of the reachable set. In Lemma \ref{lemma: prob setup: boundary mapping}, it was shown that the boundary of the backwards reachable set cannot come from the interior of the terminal set $\mathcal{X}_T$. This applies in a recursive manner, meaning, the boundary of the backwards reachable set at time $t$ cannot come from the interior of the set at a time $\tau > t$. Thus, propagating the tree from nodes that lie in the interior of the backwards reachable set seems to be a wasted expense. 

\begin{figure}[h!]
\centering
\includegraphics[width = 0.75\textwidth]{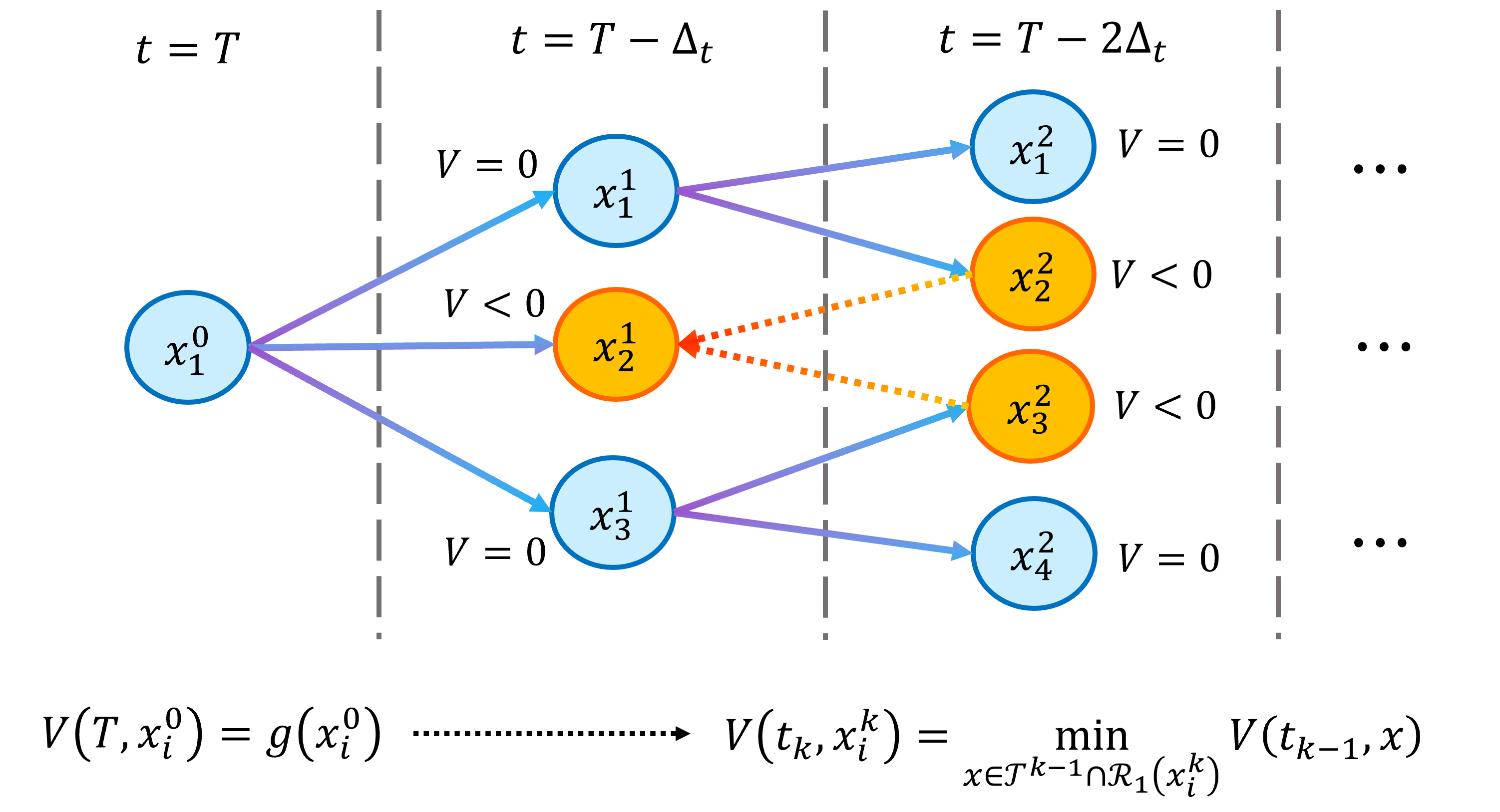}
    \caption{Diagram depicting tree structure and value function computation as described by Algorithm \ref{Alg:1}. In orange are nodes that have a strictly negative value, which are candidates for pruning.}
    \label{fig: algorithm: DP tree structure with pruning}
\end{figure}

Interior nodes can be identified by points for which $V(t, x) < 0$, and are candidates for pruning. Algorithm \ref{Alg:1} could potentially be modified in Steps 2 to 5 with a pruning step that removes nodes in $\mathcal{T}^k$ if $V(t_k, x^k_{i_j}) < -\epsilon$ for some chosen tolerance $\epsilon > 0$. However, removing \emph{all} such nodes may be problematic as \eqref{eq: algorithm: value function propagation} minimizes over all nodes in the previous level $k-1$. If none of the nodes in the previous level have a value less than $-\epsilon$, then it is not possible for any node on level $k$ to have a value less than $-\epsilon$ and hence interior nodes can no longer be identified. This issue is highlighted in Figure \ref{fig: algorithm: DP tree structure with pruning}. If nodes $x^2_2$ and $x^2_3$ can reach node $x^1_2$ using an input $u\in\mathbb{U}$ and $V(T-\Delta_t, x^1_2) < 0$, then it must also be the case that the value at these nodes is also strictly negative, and are hence interior nodes. However, if $x^1_2$ is pruned from the previous level, then $x^2_2$ and $x^2_3$ can no longer be identified as interior nodes. In general, some interior nodes may need to be stored so that nodes on further tree levels can be identified as being interior. These nodes can then be stored more sparsely by removing some of its members without impacting knowledge of the boundary of the reachable set. 

There is, however, a case where interior nodes can be identified without needing to store interior nodes in previous tree levels. In particular, a geometric condition for identifying interior nodes can be used. To this extent, first consider the short result below. 
\begin{lemma} Assume that the value function $v\in\mathcal{C}\left([0,T]\times\mathbb{R}^n\,;\,\mathbb{R}\right)$ in \eqref{eq: prob setup: HJB PDE BRS} is convex in $x$ for all $t\in[0,T]$ and let $\{x_i\}_{i\in\{1,\ldots,k\}}$ be a finite set of points contained in $\mathcal{G}(\tau)$ for some time $\tau \in [0,T]$. Then,
\begin{equation}
    \text{conv}\left(\{x_i\}_{i\in\{1,\ldots,k\}}\right) \subseteq \mathcal{G}(\tau). \label{eq: algorithm: conv hull subset}
\end{equation}
\label{lemma: algorithm: conv hull subset}
\end{lemma}
\begin{proof}
Let $\hat{x}\in\text{conv}\left(\{x_i\}_{i\in\{1,\ldots,k\}}\right)$, then $\hat{x} = \sum^k_{i=1}\lambda_i x_i$ for $\lambda_i \geq 0$ and $\sum^k_{i=1}\lambda_i = 1$. By convexity of $v$, we have that for any $\tau \in [0,T]$,
\begin{equation}
    v\left(T-\tau,\hat{x}\right)=v\left(T - \tau, \sum^k_{i=1}\lambda_i x_i\right) \leq \sum^k_{i=1}\lambda_i v\left(T - \tau, x_i\right).
    \label{eq: algorithm: conv hull alg proof}
\end{equation}
Since $x_i \in \mathcal{G}(\tau)\implies v\left(T - \tau, x_i\right) \leq 0$, the right-hand side of \eqref{eq: algorithm: conv hull alg proof} is non-positive. This then implies $v\left(T-\tau, \hat{x}\right)\leq0$, thus $\hat{x}\in\mathcal{G}(\tau)$ and \eqref{eq: algorithm: conv hull subset} holds. 
\end{proof}
It follows from Lemma \ref{lemma: algorithm: conv hull subset} that for convex value functions, points that lie interior to the convex hull of the nodes on any given tree level $k$ must also be interior to the backwards reachable set $\mathcal{G}(t_k)$ (ignoring integration errors of \eqref{eq: algorithm: node generation}). This implicitly identifies nodes for which $V < 0$.

A modification to Algorithm \ref{Alg:1} is presented in Algorithm \ref{Alg: convex hull pruning} where assumptions on convexity are used to avoid explicit computation of the value function for identifying interior points. Here $\Tilde{\mathcal{T}}^k$ denotes the tree at level $k$ containing a set of nodes $\{\Tilde{x}^k_{i}\}_{i\in\{1,\ldots,\Tilde{n}_k\}}$ prior to any pruning. If a node in $\Tilde{\mathcal{T}}^k$ lies interior to the convex hull of $\Tilde{\mathcal{T}}^k$, it is not added to the tree level $\mathcal{T}^k$, and is thus not propagated in the overall tree structure. 
\begin{algorithm}
\caption{Tree Structure Algorithm with Convex Hull Pruning}
\label{Alg: convex hull pruning}
\begin{algorithmic}[1]
\Require{$\Delta_t, \Bar{\mathbb{U}} \doteq\{u_1, \ldots, u_{n_u}\}\subset \mathbb{U}$.} 
\initialise{ Discretize $\partial\mathcal{X}_T$ to $\{x^0_i\}_{i\in\{1,\ldots,n_0\}}$}
\For{$k=1, \ldots, N$}
\State $\Tilde{x}^k_{i_j} = x^{k-1}_i - \Delta_t f(x^{k-1}_i, u_j),\qquad \forall i\in\{1,\ldots,n_{k-1}\},\; \forall j\in\{1,\ldots,n_u\}$
\State $\Tilde{\mathcal{T}}^k \leftarrow \Tilde{x}^k_{i_j}, \qquad \forall i\in\{1,\ldots,n_{k-1}\},\; \forall j\in\{1,\ldots,n_u\}$
\State $n_k \leftarrow 0$
\ForAll{$\Tilde{x}^k_{i_j} \in \Tilde{\mathcal{T}}^k$}
\If{$\Tilde{x}^k_{i_j}\in \partial\text{conv}\left(\Tilde{\mathcal{T}}^k\right)$}
\State $\mathcal{T}^k\leftarrow \Tilde{x}^k_{i_j}$
\State $n_k \leftarrow n_k + 1$
\EndIf
\EndFor
\EndFor\\
\Return $\mathcal{T}^{N}$
\end{algorithmic}
\end{algorithm}

\begin{remark} The value function $v\in\mathcal{C}\left([0,T]\times\mathbb{R}^n\,;\,\mathbb{R}\right)$ in \eqref{eq: prob setup: HJB PDE BRS} is convex in $x$ for all $t\in[0,T]$, if $\mathbb{U}$ is a convex set, $g:\mathbb{R}^n\times\mathbb{R}$ is a convex function, and \eqref{eq: prob setup: nonlinear system} is described by the linear dynamics
\begin{equation}
    \dot{x}(t) = Ax(t) + Bu(t).
    \label{eq: algorithm: linear system dyn}
\end{equation} 
To see this, we note that the viscosity solution of \eqref{eq: prob setup: HJB PDE BRS} corresponds to the value function
\begin{equation}
     v(t,x) \doteq \inf_{u(\cdot)\in\mathcal{U}}g\left(\varphi(T;t,x,u(\cdot))\right).
\end{equation}
Solutions $\varphi$ of \eqref{eq: algorithm: linear system dyn} take the form 
\begin{equation}
    \varphi(T; t, x, u(\cdot)) = e^{A(T-t)}x + \int^T_te^{A(T-s)}Bu(s) ds,
\end{equation}
which is affine in both $x$ and $u(\cdot)$. Thus, $g$ must be convex in both arguments $x$ and $u(\cdot)$. Furthermore, convexity of $\mathbb{U}$ implies convexity of $\mathcal{U}$, then standard results from convex analysis (see e.g. \cite{rockafellar1970convex}) can be used to show that $v$ must be convex in $x$. In particular, if a function $f:\mathcal{X}\times\mathcal{Y}\rightarrow\mathbb{R}$ is convex in both of its arguments with $\mathcal{Y}$ being a convex set, then the function $\bar{f}(x)\doteq \inf_{y\in\mathcal{Y}}f(x,y)$ is convex in $x$. This is a special case of Theorem 7.4.13 in \cite{CS04} where the running cost $h$ in \eqref{eq: prob setup: value function} is omitted.
\label{remark: algorithm: linear system convex value function}
\end{remark} 
\begin{remark}
If the system \eqref{eq: prob setup: nonlinear system} is control affine, i.e. the flow field $f:\mathbb{R}^n\times\mathbb{U}\rightarrow\mathbb{R}$ can be decomposed as
\begin{equation*}
    f(x(t),u(t)) = f_1(x(t)) + f_2(x(t))u,
\end{equation*}
and $\mathbb{U}$ is an ellipsoidal set, then, nodes that lie along the boundary $\partial\mathcal{G}(T)$ must originate from $\partial\mathcal{X}_T$ under a control satisfying $u(t)\in\partial\mathbb{U}$ for all $t\in[0,T]$. This can be demonstrated by noting that points along $x\in\partial \mathcal{G}(T)$ can only reach $\partial\mathcal{X}_T$ (see Lemma \ref{lemma: prob setup: boundary mapping}) and they do so under a control law satisfying $u(t)\in \argmin_{u\in\mathbb{U}}\iprod{\nabla v}{f(x,u)}$. If \eqref{eq: prob setup: nonlinear system} is control affine, then the control is minimal with respect to a linear function in $u$. Thus, if $\mathbb{U}$ is ellipsoidal, the minimising control lies along $\partial\mathbb{U}$ (see standard results on support functions over ellipsoidal sets e.g. \cite{rockafellar1970convex}). 
\label{remark: algorithm: input boundary for control affine sets with ellipsoidal inputs}
\end{remark}



\section{Numerical Examples}
We use Algorithm \ref{Alg: convex hull pruning} to compute the backwards reachable set $\mathcal{G}(T)$ for two example systems. The first example looks at a linear system with two states and two inputs, whilst the second example looks at a nonlinear system model for a DC motor consisting of three states. For comparison purposes, the backwards reachable sets are also computed using an off-the-shelf toolbox provided by \cite{mitchelltoolbox}. This toolbox contains a grid-based, finite-difference scheme for numerically evaluating the value function of Theorem \ref{theorem: prob setup: HJB PDE General}. By omitting the running cost $h:\mathbb{R}^n\times\mathbb{U}\rightarrow \mathbb{R}$ and extracting the zero level set $\left\{x\in\mathbb{R}^n\,|\,V(0,x)= 0\right\}$, we obtain the boundary of the backwards reachable set. Throughout the following, we will use $\mathcal{G}_\mathcal{T}(T)$ to denote the backwards reachable set computed by taking the convex hull of the final tree nodes $\mathcal{T}^{N}$ in Algorithm \ref{Alg: convex hull pruning} and $\mathcal{G}_{FD}(T)$ to denote the backwards reachable set computed via the toolbox of \cite{mitchelltoolbox}. In the following examples, we make use of the `qhull' algorithm (see \cite{barber1996quickhull}) to compute the convex hull in Step 7 of Algorithm \ref{Alg: convex hull pruning}, which is available via standard routines in \texttt{MATLAB}.

\subsection{Numerical Example 1: Linear System with Two States}
Consider the linear time-invariant system described by
\begin{equation}
    \Dot{x}(t) = \begin{bmatrix}0&1\\1&0\end{bmatrix}x(t) + \begin{bmatrix}1&0\\0&1\end{bmatrix}u(t), \quad \forall t\in(0, T),
\label{eq: numerical ex: 2D system description}
\end{equation}
where $x(t)=[x_1(t), x_2(t)]^T\in\mathbb{R}^2$ is the state and $u(t)\in\mathbb{U}\subset\mathbb{R}^2$ is the input at time $t$, with a terminal condition $x(T)\in\mathcal{X}_T\subset \mathbb{R}^2$. We take the input constraint set $\mathbb{U}$ and the terminal set $\mathcal{X}_T$ to be ellipsoidal sets given by
\begin{equation}
    \mathbb{U} \doteq \mathcal{E}\left(\begin{bmatrix}0\\1\end{bmatrix},\begin{bmatrix}4 & 0\\0&1\end{bmatrix} \right),\qquad  \mathcal{X}_T \doteq \mathcal{E}\left(\begin{bmatrix}0\\0\end{bmatrix},\begin{bmatrix}0.01 & 0\\0&0.01\end{bmatrix} \right).
    \label{eq: numerical ex: 2D system ellipsoidal sets}
\end{equation}
In the notation of \eqref{eq: prob setup: initial and terminal set}, we have the terminal state cost $g(x) = 100\norm{x}^2_2 -1$. 

To implement Algorithm \ref{Alg: convex hull pruning} for \eqref{eq: numerical ex: 2D system description}, the terminal set $\mathcal{X}_T$ was discretised into a set of $n_0 = 20$ nodes, uniformly distributed along its boundary. Since \eqref{eq: numerical ex: 2D system description} is control affine and $\mathbb{U}$ is ellipsoidal, the set of optimal inputs that generate nodes along the boundary of the backwards reachable set must come from the boundary of $\mathbb{U}$ (see Remark \ref{remark: algorithm: input boundary for control affine sets with ellipsoidal inputs}). Accordingly, $\mathbb{U}$ in \eqref{eq: numerical ex: 2D system ellipsoidal sets} was discretised into a set of $n_u = 15$ points distributed about its boundary. In particular, the following discretised input set was used:
\begin{equation}
    \Bar{\mathbb{U}} \doteq \left\{\begin{bmatrix}2 &0 \\ 0 & 1\end{bmatrix}w + \begin{bmatrix}0 \\ 1\end{bmatrix}\,\Bigg|\, w = \left[\sin{\left(\frac{2\pi k}{n_u}\right)},\; \cos{\left(\frac{2\pi k}{n_u}\right)}\right]^T, \; k\in\{1,\cdots,n_u\}\right\}.
\end{equation}

\begin{figure}[h!]
  \centering
  \begin{minipage}[b]{0.5\textwidth}
    \includegraphics[width=1\columnwidth]{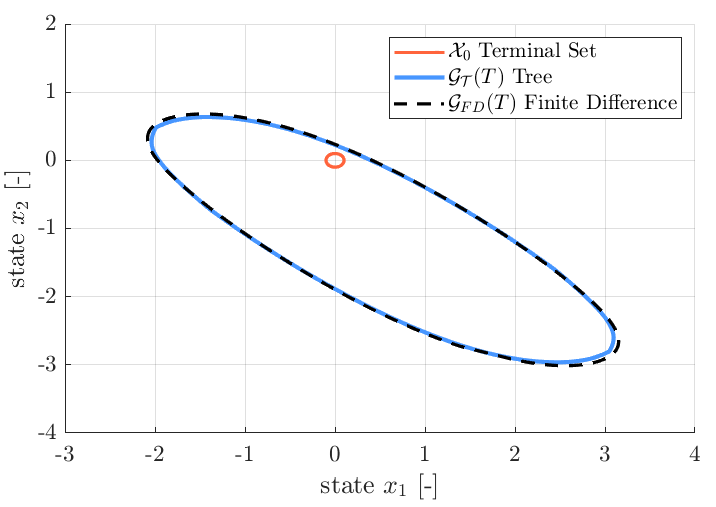}
    \setlength{\abovecaptionskip}{0pt}
    \setlength{\belowcaptionskip}{0pt}
    \caption{Comparison of backwards reachable sets for the system dynamics in \eqref{eq: numerical ex: 2D system description} computed via a finite difference approach and Algorithm \ref{Alg: convex hull pruning}.}
  \label{fig: numerical ex: 2D reachable sets}
  \end{minipage}
  \hfill
  \begin{minipage}[b]{0.47\textwidth}
    \includegraphics[width=1\columnwidth]{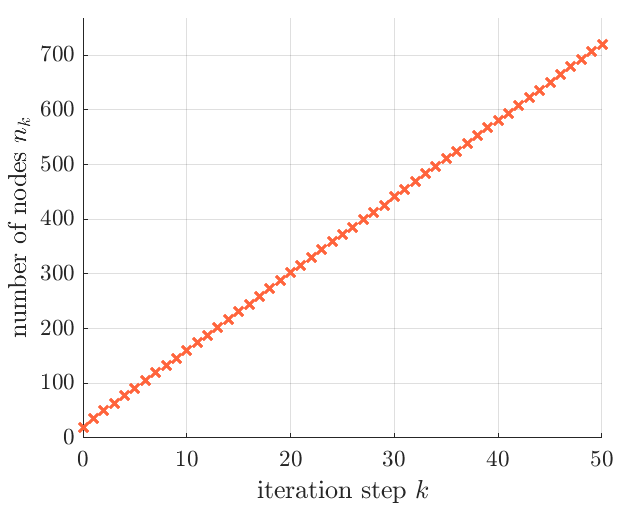}
        \setlength{\abovecaptionskip}{0pt}
    \setlength{\belowcaptionskip}{0pt}
    \caption{Number of nodes used to describe $\mathcal{G}_{\mathcal{T}}(T)$ for the system dynamics \eqref{eq: numerical ex: 2D system description} in each iteration of Algorithm \ref{Alg: convex hull pruning}.}
    \label{fig: numerical ex: 2D nodes}
  \end{minipage}
\end{figure}

To compute $\mathcal{G}_{FD}(T)$ via a finite-difference scheme, a grid of $200\times 200$ nodes was used, noting that coarser grids were noticeably under-approximating compared to expected results. The backwards reachable sets as computed by Algorithm \ref{Alg: convex hull pruning} and the finite-difference scheme are depicted in Figure \ref{fig: numerical ex: 2D reachable sets} for $T = 1$s. Also displayed in Figure \ref{fig: numerical ex: 2D nodes} is the number of nodes $n_k$ used to describe the boundary of the reachable set in each iteration of Algorithm \ref{Alg: convex hull pruning}, noting that a time-step of $\Delta_t = 0.02$s was selected.

Using a four-core Intel{\textregistered} Core\textsuperscript{\texttrademark} i7-1065G7 CPU, computation times in \texttt{MATLAB} for computing $\mathcal{G}_{\mathcal{T}}(T)$ and $\mathcal{G}_{FD}(T)$ were $6.11$s, and $34.1$s, respectively. The internal area of the sets $\mathcal{G}_{\mathcal{T}}(T)$ and $\mathcal{G}_{FD}(T)$ were computed via a trapezoidal integration scheme and were found to be 8.50 units\textsuperscript{2} and 8.68 units\textsuperscript{2}, respectively. We observe that Algorithm \ref{Alg: convex hull pruning} offers markedly lower computation times at a fairly small expense to the captured area, which can be improved by increasing the size of the discretised set $\Bar{\mathbb{U}}$ if needed. It is interesting to note that despite Step 3 of Algorithm \ref{Alg: convex hull pruning} producing 15 `candidate' nodes for each node in the previous tree level, the resulting number of nodes does not grow exponentially after pruning. The final tree level returned by Algorithm \ref{Alg: convex hull pruning} has only $720$ nodes as compared with the 40,000 grid points used in the finite difference scheme. Moreover, if a longer horizon $T$ was selected, then a new grid, covering a larger domain, may be required for the finite-difference scheme. This is not necessary for Algorithm \ref{Alg: convex hull pruning}.

In this example, the value function $v:[0,T]\times\mathbb{R}^n\rightarrow\mathbb{R}$ as described in Theorem \ref{theorem: prob setup: HJB equation for BRS} can be shown to be convex (see Remark \ref{remark: algorithm: linear system convex value function}) thus the backwards reachable set $\mathcal{G}(T)$ of \eqref{eq: numerical ex: 2D system description} is also convex, since sub-level sets of convex functions are convex. Consequently, taking a convex hull of the nodes contained in each tree level $\{x^k_i\}_{i\in\{1,\cdots,n_k\}}$ results in the backwards reachable set computed via Algorithm \ref{Alg: convex hull pruning} being an inner approximation of $\mathcal{G}(T)$ (ignoring integration errors in Step 3 of Algorithm \ref{Alg: convex hull pruning}).

\subsection{Numerical Example 2: DC Motor}
In our next example, we consider a typical nonlinear system model for a DC motor, which consists of three states: the rotor angle $x_1$, the rotor angular velocity $x_2$, and the armature current $x_3$. The input to the system is the supplied voltage $u$. An example DC motor model with arbitrarily selected system parameters is given by
\begin{equation}
    \dot{x}(t) = \begin{bmatrix}x_2(t) \\ -10\sin{(x_1(t))} - \text{sign}(x_2(t))x^2_2(t) + 5x_3(t) \\ -10x_2(t) + 50x_3(t)\end{bmatrix} + \begin{bmatrix}0\\0\\50\end{bmatrix}u(t), \quad \forall t\in (0,T), \label{eq: numerical ex: DC motor dyn}
\end{equation}
where $x(t)=[x_1(t), x_2(t), x_3(t)]^T\in\mathbb{R}^3$ is the state and $u(t)\in\mathbb{U}\subset\mathbb{R}$ is the input at time $t$, with a terminal condition $x(T)\in\mathcal{X}_T\subset \mathbb{R}^2$. We take the input constraint set $\mathbb{U}$ and the terminal set $\mathcal{X}_T$ to be ellipsoidal sets given by
\begin{equation}
    \mathbb{U} \doteq \mathcal{E}\left(0, 4\right), \qquad  \mathcal{X}_T \doteq \mathcal{E}\left(\begin{bmatrix}\frac{\pi}{2}\\0\\0\end{bmatrix},0.04\mathbb{I}_3 \right).
    \label{eq: numerical ex: DC motor dyn ellipsoidal sets}
\end{equation}

Here, the input constraint set is equivalent to the interval $\mathbb{U} = [-2, 2]$ and the terminal set $\mathcal{X}_T$ is a small ball around the unstable equilibrium (with zero input) $\bar{x} = $ $\left[\textstyle{\frac{\pi}{2}}, 0, 0\right]^T$ of \eqref{eq: numerical ex: DC motor dyn}. Since \eqref{eq: numerical ex: DC motor dyn} is control affine and $\mathbb{U}$ is an interval constraint, optimal controls that generate nodes along the boundary of the reachable set $\mathcal{G}(T)$ must lie on the extremal points of $\mathbb{U}$ (see Remark \ref{remark: algorithm: input boundary for control affine sets with ellipsoidal inputs} and Remark 3.1 of \cite{AFS19}). A natural discretization of $\mathbb{U}$ is then given by $\Bar{\mathbb{U}} \doteq \left\{-2, 2\right\},$
which contains the optimal control for evolving the boundary of the backwards reachable set. Additionally, to compute $\mathcal{G}(T)$ of \eqref{eq: numerical ex: DC motor dyn} using Algorithm \ref{Alg: convex hull pruning}, the terminal set $\mathcal{X}_T$ was discretised into a set of $n_0 = 84$ nodes, uniformly distributed along its boundary. 

Likewise with the first numerical example, the finite-difference scheme implemented by the toolbox of \cite{mitchelltoolbox} was used as comparison for computing the backwards reachable set. In the finite-difference scheme, a grid of $101\times101\times101$ nodes was used. The backwards reachable set as computed by Algorithm \ref{Alg: convex hull pruning} and the finite-difference scheme are depicted in Figure \ref{fig: numerical ex: DC motor reachable sets} for $T = 0.02$s with a time discretisation of $\Delta_t = 0.4$ms used in Algorithm \ref{Alg: convex hull pruning}. Figure \ref{fig: numerical ex: DC motor nodes} displays the number of nodes $n_k$ used to describe the boundary of the reachable set in each iteration of Algorithm \ref{Alg: convex hull pruning}. 
\begin{figure}[h!]
  \centering
  \begin{minipage}[b]{0.53\textwidth}
    \includegraphics[width=1\columnwidth]{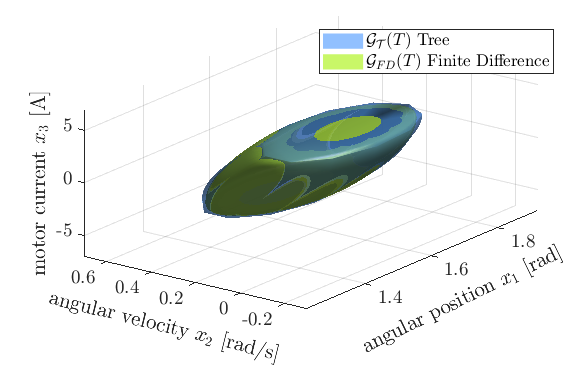}
    \setlength{\abovecaptionskip}{0pt}
    \setlength{\belowcaptionskip}{0pt}
    \caption{Comparison of backwards reachable sets for the system dynamics in \eqref{eq: numerical ex: DC motor dyn} computed via a finite difference approach and Algorithm \ref{Alg: convex hull pruning}.}
  \label{fig: numerical ex: DC motor reachable sets}
  \end{minipage}
  \hfill
  \begin{minipage}[b]{0.45\textwidth}
    \includegraphics[width=1\columnwidth]{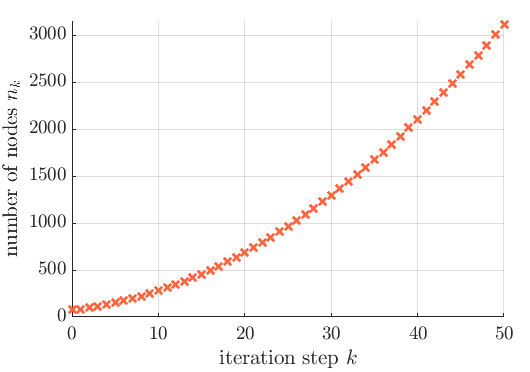}
        \setlength{\abovecaptionskip}{0pt}
    \setlength{\belowcaptionskip}{0pt}
    \caption{Number of nodes used to describe $\mathcal{G}_{\mathcal{T}}(T)$ for the system dynamics \eqref{eq: numerical ex: DC motor dyn} in each iteration of Algorithm \ref{Alg: convex hull pruning}.}
    \label{fig: numerical ex: DC motor nodes}
  \end{minipage}
\end{figure}

Computation times in \texttt{MATLAB} for computing $\mathcal{G}_\mathcal{T}(T)$ and $\mathcal{G}_{FD}(T)$ were 4.984s and 456.8s, respectively. Numerical integration of the sets $\mathcal{G}_\mathcal{T}(T)$ and $\mathcal{G}_{FD}(T)$ produced internal volumes of 1.059 units\textsuperscript{3} and 1.055 units\textsuperscript{3}, respectively. Again, the computation times are noticeably smaller for Algorithm \ref{Alg: convex hull pruning} whilst still producing a backwards reachable set of comparable volume and shape to the finite-difference scheme. The backwards reachable set appears to be convex, and from Lemma \ref{lemma: algorithm: conv hull subset}, this would suggest that $\mathcal{G}_\mathcal{T}(T)$ should be an inner approximation of $\mathcal{G}_{FD}(T)$, and thus should be contained inside the set in green in Figure \ref{fig: numerical ex: DC motor reachable sets}. However, integration errors in both Algorithm \ref{Alg: convex hull pruning} and the finite-difference scheme may cause the boundary of $\mathcal{G}_\mathcal{T}(T)$ to lie outside of $\mathcal{G}_{FD}(T)$. 

The number of nodes in the final tree level $\mathcal{T}^{N}$ of Algorithm \ref{Alg: convex hull pruning} was 3111, and like the first numerical example, the nodes seem to exhibit sub-exponential growth in the number of iteration steps $k$ due to the pruning of interior points. We note that if a long horizon is used, then the number of nodes may grow to be quite large. In such a case, additional pruning may be needed, for instance, removing nodes that lie sufficiently close to other nodes in the same tree level as is suggested in \cite{AFS19}.

\section{Conclusions and future works}

In this work, we have proposed a tree structure algorithm to compute reachable sets using the Hamilton-Jacobi approach. Our method computes the tree backwards in time, starting from the terminal set and using a finite set of controls. To mitigate issues associated with the exponential increase in the cardinality of the tree, we have introduced a pruning strategy based on geometric considerations. In fact, at each time level of the tree, we neglect all the nodes that lie in the interior of the convex hull of that level set. 

In our numerical examples, we have shown how the algorithm compares to a finite-difference approach for a 2D linear and a 3D nonlinear system. Our method provides very accurate results with a measurable speed up in terms of computational time.

This is, to the best of the authors' knowledge, the first approach which uses a tree structure algorithm for reachable set computation. To further validate our approach, we will consider higher dimensional systems driven by applications. For instance, guidance of aircraft, collision avoidance of multi-agent systems, and control of systems described by partial differential equations. 

In the future, we would like to generalize our method by relaxing our convexity assumption, which was crucial for validating our approach, and instead consider semi-concave value functions. This would then allow us to retain the same degree of accuracy for a larger class of systems. Finally, it will also be of interest to be able to construct a safety-based control law so as to guarantee that we stay within the backwards reachable set for all time. 
\section*{Acknowledgments}
AA wants to acknowledge the {\em Overseases Mobility program} financed by Università Ca' Foscari Venezia. Funding for this research was also supported through an Australian Research Council Linkage Project grant (Grant number: LP190100104), an Asian Office of Aerospace Research and Development grant (Grant number: AOARD22IOA074), and the Australian Commonwealth Government through the Ingenium Scholarship. Acknowledgement is also given to BAE systems as a collaborator in the aforementioned research grants.


\bibliographystyle{elsarticle-num}
\bibliography{bibliography}
\end{document}